\documentclass[preprint]{imsart}
\RequirePackage[OT1]{fontenc}
\RequirePackage{amsthm,amsmath,amssymb,enumerate}
\RequirePackage[authoryear]{natbib}
\RequirePackage[colorlinks,citecolor=blue,urlcolor=blue]{hyperref}
\usepackage{graphicx}
\usepackage{tikz}

\def\E{{\mathbb E}}
\def\P{{\mathbb P}}
\def\R{{\mathbb R}}
\def\Z{{\mathbb Z}}
\def\V{{\mathbb V}}

\def\N{{\mathbb N}}

\def\V{{\mathcal V}}

\def\Vtoi{{\V_{ i }}}
\def\PF{{\mathcal P}_f}

\newtheorem{theo}{Theorem}
\newtheorem{prop}{\indent Proposition}
\newtheorem{lem}{\indent Lemma}
\newtheorem{defin}{\indent Definition}

\newtheorem{rem}{\indent Remark}
\newtheorem{ass}{Assumption}
\newtheorem{cor}{\indent Corollary}

\begin{document}
\begin{frontmatter}
\title{Phase transition for infinite systems of spiking neurons.}
\runtitle{Phase transition for systems of interacting neurons.}

\author{\fnms{P.A.} \snm{Ferrari}\thanks{Universidad de Buenos Aires, pferrari@dm.uba.ar}
},
\author{\fnms{A.} \snm{Galves}\thanks{Universidade de S\~ao Paulo, galves@usp.br}},
\author{\fnms{I.} \snm{Grigorescu}\thanks{University of Miami, igrigore@math.miami.edu}
},
\author{\fnms{E.} \snm{L\"ocherbach}\thanks{Universit\'e Paris Seine, eva.loecherbach@u-cergy.fr}}

\thankstext{t1}{June 28, 2018}
\runauthor{P. Ferrari et al.}

\begin{abstract}
We prove the existence of a phase transition for a stochastic model of interacting neurons. The spiking activity of each neuron is represented by a point process having rate $1 $ whenever its membrane potential is larger than a threshold value. This membrane potential evolves in time and integrates the spikes of all {\it presynaptic neurons} since the last spiking time of the neuron.  When a neuron spikes, its membrane potential is reset to $0$ and simultaneously, a constant value is added to the membrane potentials of its postsynaptic neurons. Moreover, each neuron is exposed to a leakage effect leading to an abrupt loss of potential occurring at random times driven by an independent Poisson point process of rate $\gamma > 0 .$
For this process we prove the existence of a value $\gamma_c$ such that the system has one or two extremal invariant measures according to whether $\gamma  > \gamma_c $ or not.
\end{abstract}



\end{frontmatter}

\section{Introduction}
In the present article we study an infinite system of interacting point processes  with memory of variable length modeling spiking neuronal networks. Our goal is to prove the existence of a phase transition for this model. 

Let us informally present the class of interacting point processes we consider. The spiking activity of each neuron in the system is represented by a point process in continuous time. Its rate changes in time and depends on the membrane potential of the neuron. This membrane potential integrates the spikes of all its {\it presynaptic neurons} since the last spiking time of the neuron.  At its spiking time, the membrane potential of the spiking neuron is reset to  $0$. Simultaneously, the membrane potentials of its {\it postsynaptic neurons} receive an additional fixed value. Moreover, each neuron is exposed to a leakage effect leading to an abrupt loss of potential occurring at random times driven by an independent Poisson point process of constant rate $\gamma .$ 

The fact  that the membrane potential is reset to $0$ at each spiking time of the neuron implies that the model can be seen as a system of interacting point processes where each of the point processes has a  memory of variable length. This is a non-trivial extension of the notion of {\it stochastic chain with memory of variable length} introduced by \cite{rissanen}. The biologically motivated fact that the memory of each neuron has variable length introduces an interesting mathematical challenge. Indeed, the jump rate of each neuron, as a function of the past, is not continuous, contrarily to what happens usually with interacting point processes having memory of infinite length, see for instance \cite{BremaudMassoulie94}. 

Our system is an extension to continuous time of the model of interacting neurons introduced in \cite{GL13}. 
Since then, several variants of this model have been discussed in the literature, see for instance \cite{AAEE15}, \cite{Touboul-Robert}, \cite{fournier2016}, \cite{Duarte_Ost:15}, \cite{Yaginuma2016}, \cite{Brochinietal:16} for a non exhaustive list of references. We refer to \cite{gl} for a review.

In the present paper we prove phase transition for a specific instantiation of this model. We consider a continuous time version of the model in which the set of all neurons is represented by the $1$-dimensional lattice and each neuron has its neighbors as post-synaptic neurons. 
In this version the membrane potential of each neuron takes only positive integer values and the associated spiking rate is given by $1$ if and only if the potential is strictly positive. Therefore, neurons in our system are either {\it quiescent}, when they have
membrane potential $0,$ or {\it active}, when they have strictly positive membrane potential. The distinction between {\it active} and {\it quiescent} states is reminiscent of \cite{WilsonCowan72},
  see also \cite{Kilpatrick2015} and \cite{Cowantalk}.

For this model, phase transition means that there exists a critical value $\gamma_c\in\, ] 0 , +\infty [ $ for the leakage parameter $\gamma$, such that for $\gamma  > \gamma_c,$ each neuron spikes a finite number of times and for $\gamma <  \gamma_{c},$ each neuron spikes infinitely many times. This is the content of our main result, Theorem \ref{theo:0}. 
To the best of our knowledge, up to now, this is the first rigorous proof of the existence of a phase transition for this model. 

To prove this we consider another system with interacting components, namely the system of spiking rates of the neurons. Let us denote $\eta_i(t)$ the spiking rate of neuron $i$ at time $t$. Then the resulting time evolution of the spiking rates of the neurons is an \emph{interacting particle system}, a Markov process $\eta(t)\in \{0,1\}^\Z$ with local interactions (\cite{Liggett1985}).  When $\eta_i ( t) $ equals one, at rate $\gamma$ it goes to 0 (this is the leakage part of the evolution). Moreover, when $ \eta_i  (t) = 1 , $ then at rate 1, neuron $i$ \emph{spikes}. When neuron $i$ spikes, its rate goes to 0 and at the same time the spiking rates of neurons $i-1$ and $i+1$ both go to 1, regardless of their values before the spiking. A quiescent neuron can become active only when a nearest neighbor neuron spikes. We show that this system has a dual in the sense of  \cite{harris1976}, see also \cite{association}. With these tools and a contour argument for the dual process, reminiscent to the one given by \cite{griffeath} for the contact process, we prove that the process $\eta_t$ has a phase transition in the following sense. There exists a critical value $\gamma_c\in\, ] 0 , +\infty [ $  such that for $\gamma  > \gamma_c,$ the system has only one trivial invariant measure in which all neurons are quiescent and no spiking activity is present, and for $\gamma < \gamma_c$, a second nontrivial  invariant measure exists according to which all neurons spike infinitely often. This implies our main result,  Theorem~\ref{theo:0}.

This paper is organized as follows. In Section \ref{sec:def}, we present the general model and state Theorem \ref{theo:0}, which is our main theorem. The associated interacting particle system and the notion of duality are discussed in Section \ref{sec:dual}, where we also prove Theorem \ref{theo:deux} on the extinction probability of the dual process. The proof of Theorem \ref{theo:0} is given in Section \ref{sec:proof}. Finally, in Section \ref{sec:4}, we discuss the particular case $\gamma = 0 $ in which an ad hoc proof of the existence of two extremal invariant measures is given.

\section{Definitions and main result}\label{sec:def}
We consider an infinite system of interacting point processes with memory of variable length. This system is described as follows. First of all, let $ I$ be a countable set. To each $i \in I, $ we attach two point processes $( N^{ \dag}_i (t), t \geq 0 )   $ and 
$ ( N^{ *}_i (t)  ,  t \geq 0  ) , $ defined on a suitable probability space $ ( \Omega , { \mathcal A} , P ) ,$ with their associated filtration $({ \mathcal F}_t)_{t \geq 0} $ where $ { \mathcal F}_t = \sigma \{ N^{ \dag}_i  ( s) , N^{ *}_i  ( s ) , 
 s \le t , i \in I \} .$ 

The first family 
$( N^{ \dag}_i ,  i \in I ) $ is composed by  i.i.d.\  Poisson processes of rate $\gamma \geq 0 .$ 

The second family $( N^{ *}_i ,  i \in I ) $ is characterized by the property that for all  $ s \le t, $  
\begin{equation}\label{eq:def1}
E  ( N^*_i ( t) - N^*_i (s)    \, | \, { \mathcal F}_s ) = \int_s^t E (\varphi_i (X_i (u) )  | \, { \mathcal F}_s ) du  ,
\end{equation}
with 
\begin{equation}\label{eq:def2}
X_i ( t) = \sum_{ j \in \Vtoi   } \int_{ ] L_i (t) , t [ } d N^*_j (s) ,
\end{equation}
for all $ t \geq 0.$ In the above formula, $\Vtoi \subset I$ is a subset of $I$ and
$$ L_i (t)  = \sup \{ s \le t : \Delta N^*_i ( s ) +  \Delta N^{\dag}_i ( s ) = 1 \} ,$$
where $ \Delta  N^*_i ( s )  =  N^*_i ( s ) -  N^*_i ( s - )$ and  $ \Delta  N^\dag_i ( s )  =  N^\dag_i ( s ) -  N^\dag_i ( s - ).$
Moreover, $ \varphi_i , i \in I, $ is a family of  rate functions. 

The purpose of our paper is to model an interacting system of spiking neurons. In neurobiological terms, $I $ is the set of neurons, and for each $i \in I, $ $X_i (t) $ represents the {\it membrane potential} of neuron $i$ at time $t $ and  $\Vtoi$ is the set of {\it presynaptic neurons} of $i. $ We interpret the atoms of $N^*_i $ as {\it spiking times} of neuron $i.$ The effect of these spikes is propagated through the system via {\it chemical synapses}, increasing the potential of neurons $j$ belonging to the set of {\it postsynaptic neurons}  $ \{ j \in I : i \in \V_{ j }  \}  .$
According to  \eqref{eq:def2}, the membrane potential $X_i (t)$ of neuron $i$ is reset to $0,$ interpreted as resting potential, after each spike of neuron $i.$ 
The atoms of $N^\dag_i $ are interpreted as {\it total leak times,} inducing an instantaneous total loss of membrane potential of neuron $i .$

We write $ X(t) := (X_i ( t ), i \in \Z ) .$  Observe that by construction the membrane potentials $X_i ( t) $ take values in the set of positive integers.

It is well known that the way the neurons are organized in the cortex is very complex, and a full understanding of the underlying structure is  still controversial within the neuroscience community. These questions are out of the scope of the present article, and we refer the reader to \cite{spornsreview} for a recent review of the issue. In what follows, we will assume that neurons are organized in the one dimensional lattice, that is,  $I = \Z . $ Moreover, we impose 

\begin{ass}\label{ass:1}

i) for all $i, $ $\Vtoi = \{ i-1, i + 1 \} .$ 

ii) for all $i, $ for all $ x \in \R ,$ $ \varphi_i  ( x) = \varphi ( x) = 1_{\{ x > 0 \} } .$ 
\end{ass}

\begin{rem}
Point i) of the above assumption implying nearest-neighbor interactions is certainly a simplification that does not apply to the brain's structure in general which is very complicated. But in simpler nervous tissues like the retina  neurons are arranged in layers inducing an interaction structure of this type (see, for instance, \cite{BS1998}). 
\end{rem}

Observe that the above choice of rate functions $\varphi_i$ implies that a neuron can be either {\it quiescent}, when it has membrane potential $0,$ or {\it active}, when it has strictly positive membrane potential. These two states are reminiscent of the ones appearing in the classical Wilson-Cowan model, cf.\ \cite{WilsonCowan72} and \cite{WilsonCowan73}. 

An important feature of this process is the existence of a phase transition, that is, the existence of a critical value $ \gamma_c$ with $ 0 < \gamma_c < \infty $ such that for all $\gamma >  \gamma_c, $ any fixed neuron will stop sending spikes after some finite time almost surely, while for $ \gamma < \gamma_c, $ for any initial configuration such that infinitely many neurons have membrane potential $X_i (0) \geq 1  , $ the system keeps sending spikes forever. More precisely, the following theorem holds.

\begin{theo}\label{theo:0} Under Assumption \ref{ass:1} and assuming that $ X_i (0) \geq 1$ for all $i \in \Z $,  there exists  $ \gamma_c$ with $ 0 < \gamma_c < \infty $ such that the following holds. For all $ i \in \Z,$
$$ \P (  N^*_i ( [ 0, \infty)) < \infty   ) = 1, \mbox{ if   $ \gamma > \gamma_c $} $$
and 
$$ \P (   N^*_i ( [ 0, \infty)) =  \infty  ) > 0  ,  \mbox{ if   $ \gamma < \gamma_c .$}$$
\end{theo}

To prove this result we introduce an auxiliary process which is a Markovian interacting particle system and which has the important property of possessing a dual. This is done in the next section.
\section{An auxiliary interacting particle system and its dual}\label{sec:dual} 
For any $ i \in \Z$ and $t \geq 0, $ let 
\begin{equation}\label{eq:etat}
 \eta_i (t) = 1_{\{ X_i (t) > 0 \} }
\end{equation}
and define $\eta (t) = (\eta_i (t), i \in \Z ) $
which is a process taking values in 
 $\{ 0,1\}^\Z  $.  

\begin{lem} 
The process $(\eta (t) , t \geq 0 ) $ is a continuous time Markov process on $ \{ 0, 1\}^\Z, $ with generator
\begin{equation}\label{eq:additive} 
L f (\xi) = \gamma \sum_{ i \in \Z}[  f ( \pi_i^\dag (\xi) ) - f (\xi ) ] + \sum_{ i \in \Z } \xi_i [ f ( \pi_i (\xi) ) - f(\xi ) ] ,
\end{equation}
 for any cylinder function $f:\{0,1\}^\Z\to \R.$ Here,  $ \pi_i^\dag ,  \pi_i : \{0, 1 \}^\Z \to \{0, 1 \}^\Z $ are maps defined as follows. 
\begin{eqnarray*}
 (\pi_i^\dag (\xi ) )_j &= &\xi_j , \mbox{ for all $ j \neq i ,$ }\\
(\pi_i^\dag (\xi ) )_i &= &0,
\end{eqnarray*}
and  
\begin{eqnarray*}
 (\pi_i (\xi ) )_j &=& \xi_j , \mbox{ for all $ j \neq i , i \pm 1,$ }\\
(\pi_i (\xi ) )_{i+1} &=& (\pi_i (\xi ) )_{i-1} = 1, \; (\pi_i (\xi ) )_{i} = 0 .
\end{eqnarray*}
\end{lem}
In the sequel, we denote $ (\eta^\xi(t), t \geq 0)  $ the process with generator \eqref{eq:additive}, starting from $\eta^\xi ( 0 ) = \xi$. 

In what follows we make use of the notions of additivity and duality.  To make this text self-contained, we briefly recall these notions from \cite{harris1976} and \cite{association}.

\begin{defin}
A map $ \pi : \{0, 1 \}^{\Z} \to \{0, 1 \}^{\Z} $ is called additive, if for all $ j \in \Z$ and for all $\xi \in \{0, 1 \}^ \Z,$ 
$$( \pi  (\xi))_j  = \sup \{  (\pi ( \delta(k)))_j : \xi_k = 1 \} .$$
\end{defin} 

We denote $ \PF ( \Z ) $ the set of finite subsets of $\Z .$ For any $ F \in \PF (\Z) , $ we define the function 
$ \theta_F : \{ 0, 1\}^\Z \to \{0, 1 \} $ as follows. For any $ \xi \in  \{0, 1 \}^{\Z} , $
$$    \theta_F (\xi) = \max \{ \xi_i : i \in F \} .$$

\begin{defin}
Let  $ \pi : \{0, 1 \}^{\Z} \to \{0, 1 \}^{\Z} $ and $ \sigma : \PF (\Z) \to \PF (\Z) . $ They are called dual if for any $\xi \in \{0, 1 \}^ \Z$ and any $F \in \PF (\Z)$ the following equality holds. 
\begin{equation}\label{eq:duall}
\theta_F ( \pi ( \xi ))  = \theta_{\sigma  (F)} (\xi ) . 
\end{equation}
\end{defin}

\begin{lem}\label{lem:duall} 
If $ \pi : \{0, 1 \}^{\Z} \to \{0, 1 \}^{\Z} $ is additive, then there exists a unique transformation $\pi^* $ on $ \PF ( \Z ) $ such that \eqref{eq:duall} is satisfied. This transformation is given by
$$ \pi^*  (\{ k \}) = \{ j \in \Z : ( \pi( \delta ( j))_k = 1 \}  $$
for all $k \in \Z $ and 
$$  \pi^* (F) = \bigcup_{k \in F}  \pi^* (\{ k \})  ,$$
for all $F \in \PF (\Z ) . $ 
\end{lem}

The above result is classical and very easy to prove, we refer the reader to \cite{harris1976} and \cite{association} for details. 

Let us now come back to the maps $\pi_i^\dag $ and $\pi_i $ that define the generator \eqref{eq:additive}. It is immediate to see that the following result holds true. 

\begin{lem}
For all $i \in Z, $ $\pi_i^\dag $ and $\pi_i $ are additive and their associated dual maps are given by 
$$
  (\pi_i^\dag)^ *  (F) =  F \setminus \{ i \}  $$
for all $ F \in \PF (\Z) $ and 
$$
  \pi_i^* (\{j\}) = \left\{ 
 \begin{array}{ll}
 \emptyset & \mbox { if } j = i , \\
 \{ i, j \} & \mbox{ if } j \in \{i-1 , i+1 \}, \\
 \{j\} & \mbox{ else}
 \end{array}
\right. $$
for all $j \in \Z .$ 
\end{lem}

Using these dual maps, we now define the generator of a  pure jump Markov process on $ \PF ( \Z )$ as follows. 
\begin{equation}\label{eq:generatordual}
\tilde L g( F) =   \gamma   \sum_{i \in \Z}[  g (  (\pi_i^\dag)^*  (F) ) - g (F ) ] + \sum_{ i \in \Z }  [ g (  \pi_i^* (F) ) - g(F ) ],
\end{equation}
for any finite set $F \subset  \Z .$
  
For any $ A \in \PF ( \Z ) $ we denote $C^A(t)$ the process with generator \eqref{eq:generatordual} and initial state $C^A(0)=A$. For $A = \{ i\}, $ we simply write $C^i (t) $ instead of $C^{\{ i\}} (t). $ 

The processes $( \eta^ \xi (t) , t \geq 0 ) $ having generator \eqref{eq:additive} and $(C^A(t), t \geq 0 ) $ satisfy the following duality property. 
\begin{theo}[Duality property]
For any $A\in \PF ( \Z) $ and $ \xi \in \{0, 1\}^{\Z}, $ we have 
\begin{equation}\label{eq:duality}
 \E  \, [\theta_A ( \eta^\xi  (t))]  = \E \, [ \theta_{ C^A (t) } (\xi ) ] . 
\end{equation} 
\end{theo}

This theorem was proven in \cite{association} to where we refer for a proof. Intuitively, \eqref{eq:duality} means that in order to determine weather $ \eta^\xi_i  (t) = 1 $ it suffices to check whether there exists a position $ j $ belonging to $ C^i  (t) $ such that $ \xi_j = 1 .$  

We define the extinction time of process $ C^i $
\begin{align}
  \tau^i = \inf \{ t \geq 0 : C^i (t) = \emptyset \} .
\end{align}
Then it follows directly from \eqref{eq:duality} that 

\begin{cor}\label{cor:1}
$$ \P [ \xi^ {\bf 1}_0 (t) = 1 ] = \P ( \tau^ 0 > t ) , $$
where $ {\bf 1 } \in \{0, 1 \}^\Z $ denotes the configuration such that $ {\bf 1}_i = 1 $ for all $ i \in \Z . $  
\end{cor}

\begin{proof}
By definition of  $\theta_{\{0 \}}  ,$ we have
$$
\P [ \xi^ {\bf 1}_0 (t) = 1 ]    = \E [ \theta_{\{ 0 \}} ( \xi^{\bf 1} (t) )] ,
$$
and by \eqref{eq:duality}, 
$$ \E [ \theta_{\{ 0 \}} ( \xi^{\bf 1} (t) )]
=  \E [ \theta_{\{ C^0 (t)  \}} ( { \bf 1} ) ] .$$
Then the assertion follows from
$$ \E [ \theta_{\{ C^0 (t)  \}} ( { \bf 1} ) ]  
= \P [ C^0 (t ) \neq \emptyset ] =  \P ( \tau^ 0 > t ) .
$$
\end{proof}

The main ingredient of the proof of Theorem \ref{theo:0} is the following result showing the existence of a critical parameter value for the dual process. 

\begin{theo}\label{theo:deux}
There exists $0 < \gamma_c < \infty  $ such that we have  for all $i \in \Z,$ 
\begin{equation}
 P( \tau^i < \infty   ) < 1 , \mbox{ if  $ \gamma < \gamma_c,    $}
\end{equation} 
and 
$$ P( \tau^i < \infty   ) = 1 , \mbox{ if  $ \gamma >  \gamma_c.   $}$$
\end{theo}

The proof uses the following lemma with the following definitions. 
We say that $ C^A (t) $ is simply connected if $|  C^A (t)| \geq 2 $ and if  $ k, j  \in C^A (t) $ with $ j < k $ implies that $ l \in C^A ( t) $ for all $ j \le l \le k .$ If $C^A(t) $ is not simply connected, we say that $C^A (t) $ is  disconnected. Let now $  \tau_A := \inf \{ t > 0 : C^A (t) \mbox{ disconnected  }  \} .$

\begin{lem}\label{prop:utile}\
  
\begin{enumerate}
\item Let $A$ be a simply connected set with $ |A| \geq 2.$ Then, for $ t < \tau_A,$ a particle belonging to $C^A (t )$ disappears at rate $ \gamma .$  
\item
Isolated particles within $C^i ( t) $ disappear at rate $ 1 + \gamma .$
\end{enumerate}
\end{lem}

\begin{proof}
Item 1.\ follows immediately from the definition of $ \pi_i^* , $ observing that $ \pi_i^* ( \{ i, i \pm 1\} ) = \{ i, i \pm 1 \} .$ Item 2.\ is evident.
\end{proof}

\begin{rem}
Experts in interacting particle systems would rephrase Item 1. of the above Lemma by saying that, when starting with a simply connected set $A,$ up to the first time of getting disconnected, $C^A ( t) $ behaves like a classical contact process having infection rate $1$ and recovery rate $\gamma.$
\end{rem}

In the sequel, we shall also need a monotonicity property with respect to $\gamma.$ Namely, write  $ (C^{i, \gamma } (t), t \geq 0 ) $ for the dual process evolving under the parameter $\gamma .$ Then the following holds. 

\begin{lem}
For $\gamma_1 < \gamma_2 , $ for any $ t \geq 0 $ we have
\begin{equation}
\label{eq:monotone}
P(  C^{i, \gamma_1  } (t) \neq \emptyset ) \geq P(  C^{i, \gamma_2  } (t) \neq \emptyset ).
\end{equation}

\end{lem}

\begin{proof}
The proof is done by a coupling of $ (C^{i, \gamma_1  } (t) , t \geq 0) $ and  $(C^{i, \gamma_2  } (t) , t \geq 0) $ which is defined as follows. Both processes are defined by using the same rate $1$ Poisson processes $ N_i (t) $ as before. Moreover, we use the rate $\gamma_1 $ process $ N_i^{ \dag, \gamma_1} (t) $ for both processes, and we add an independent rate $\gamma_2 - \gamma_1 $ process $N_i^{ \dag, \gamma_2 - \gamma_2 } (t)  $ in order to realize the right total leakage rate $ \gamma_2$ for $(C^{i, \gamma_2  } (t) , t \geq 0) . $ All jumps induced by  $N_i^{ \dag, \gamma_1} $ and by $ N_i$ are therefore common jumps of both process, and at jump times of $N_i^{ \dag, \gamma_2 - \gamma_2 }  , $ a transformation $ (\pi_i^\dag)^* $ is applied only to $C^{i, \gamma_2  }   $ implying that 
$$ C^{i, \gamma_2  } (t)  \subset C^{i, \gamma_1  } (t) $$
for all $ i $ and for all $ t.$ From this the result follows. 
\end{proof}

We are now able to give the 
\begin{proof}[Proof of Theorem \ref{theo:deux}]
We first show that $\tau^i < \infty $ almost surely for sufficiently large $ \gamma.$ For that sake, observe that we can couple  $|C^i (t)| $ with a branching process $( Z (t), t \geq 0)$ having infinitesimal generator 
$$ L f (n) = n\left[  ( f(n+1 ) - f(n) ) + \gamma  ( f(n-1) - f(n))\right] , n \in \N ,$$
for any bounded test function $ f $ defined on $\N. $ 

The coupling is done in the following way. We start with $ Z (0 ) = 1 .$ Every time a transformation $ \pi^ \dag_j $ with $ j \in C^i (t) $ appears, the process $Z (t) $ decreases by $1.$ Any time a transformation  $ \pi_j^ * $ with $ j \in C^i (t) $ appears, we increase the value of $Z(t) $ by $1. $  

With this coupling, $ |C^i (t)| \le  Z (t) $ almost surely for all $t \geq 0.$ It is well-known that $ \gamma \geq 1 $ implies $ P ( \lim_{t \to \infty } Z (t) = 0 ) = 1, $ implying that $ \gamma_c \le 1.$  

We now show that for all $ \gamma $ such that $ \gamma \le \gamma_c,$ $\tau^i = + \infty $ with positive probability. The proof relies on the classical graphical construction of $ C^i (t) $ which has been introduced by \cite{harris1978}. We work within the space-time diagram $ \Z \times [0, \infty[.$ For each site $ i \in \Z, $ we consider two independent Poisson processes $N_i  $ and $N_i^\dag.  $  $N_i  $ has intensity $ 1$ and $ N_i^\dag $ has intensity $\gamma .$ Let us  write  $ (T_{i, n} )_{ n \geq 1 } $ and $ (T_{i, n }^\dag)_{n\geq 1 }$ for their respective jump times. The processes associated to different sites are all independent. We draw 
\begin{itemize}
\item
arrows pointing from $ (i-1 , T_{i, n}  )$ to $ (i , T_{i, n} ) $ and from $ (i+ 1 , T_{i, n } ) $ to $ (i, T_{i, n } ) , $ for all $ n \geq 1, i \in \Z  ;$
\item 
$\delta$'s at all $ (i, T_{i, n }^\dag ) , $ for all $ n \geq 1, i \in \Z.$ 
\end{itemize}
By convention, we suppose that time is going up. 

In this way we obtain a random graph $ {\mathcal P}.$ We say that there is a path from $ (i, 0) $ to $ (j , t) $ in $ {\mathcal P}$ if there is a chain of upward vertical and directed horizontal edges which goes from $ (i, 0) $ to $ (j , t) $ without passing through a $\delta.$ Notice that 
$$ C^i (t)  = \{ j : \mbox{ there exists a path from $(i, 0) $ to $ (j, t ) $ }\}.$$ 

It is clear that $\tau^i < \infty $ if and only if $ C^i := \bigcup_{t \geq 0} C^i ( t) $ is a finite set. Inspired by classical contour techniques we will show that $P( \tau^i < \infty ) =  P( |C^i| < \infty  ) < 1$ for sufficiently small values of $\gamma .$

This is done as follows. On $ |C^i | < \infty , $ we draw the contour of $C^i $ following \cite{griffeath}. For this sake, we embed $ \Z \times \R_+ $ in $ \R \times \R_+ $ and define 
$$ E := \{ (y, t ) : \| y - j \| \le \frac12, \mbox{ for some } j \in C^i (t) , t \geq 0 \} .$$
Moreover, we write $\tilde E $ for the set that one obtains from $E$ by filling in all holes of $E.$ We write $ \Gamma $ for the boundary of $ \tilde E , $ oriented clockwise. Starting from $ (i- \frac12, 0 ) , $ $\Gamma 
$ consists of $4n $ alternating vertical and horizontal edges for some $ n \geq 1 $ which we encode as a succession of direction vectors $(D_1, \ldots, D_{2n } ).$ Each of the $D_i$ can be one the seven triplets 
$$ dld, drd, dru, ulu, uru, urd, dlu, $$
where $d,u, l$ and $r$ stand for down, up, left and right, respectively. We start at $ (i-\frac12, 0) $ and proceed clockwise around the curve. 

Writing $N (dld) , \ldots $ for the number of appearances of the different direction vectors, we have that $ N(dlu) = 1 $ ($dlu$ is the last triplet of $ \Gamma $ which appears exactly one single time) and 
$$ N(dld) + N(ulu) = n-1, \; N(drd) + N(dru) + N (uru) + N(urd) = n, $$
$N(urd) = N( dru)+1.$ Moreover, 
$$ N(urd) \le n/2 +1 .$$ 

 When following the contour of the curve $\Gamma, $ at each step the choice of a given direction vector depends on which one of the relevant Poisson processes affecting this edge occurs first.  We first observe that the occurrence of either $drd$ or $dru $ or $uru$ can be upper bounded by $\gamma. $ This is due to the fact that in these events, the dying particles are non-isolated and that by Item 1.\ of Lemma \ref{prop:utile},  non isolated particles die at rate $\gamma.$ The associated probability is $ \frac{\gamma}{3 + \gamma } $ or $ \frac{\gamma}{2 + \gamma },$ depending on whether the dying particle has two neighbors or just a single one. In any case, it can be upper bounded by $ \gamma .$  Finally, recall that by the first part of the proof, we can restrict ourselves to the case  $ \gamma < 1.$ We upper bound the probabilities  of the  remaining directions by $1.$ Therefore we obtain the following list of upper bounds 
$$
\begin{array}{cl}
dld & \mbox{ occurs with probability at most } 1 \\
drd & \mbox{ occurs with probability at most }  \gamma \\
dru & \mbox{ occurs with probability at most }  \gamma \\
ulu & \mbox{ occurs with probability at most  }  1 \\
uru & \mbox{ occurs with probability at most }  \gamma \\
urd & \mbox{ occurs with probability at most }  1 \\
dlu & \mbox{ occurs with probability at most }  1 .
\end{array}
$$

In the above prescription, we have upper bounded the probability of dying for isolated particles, represented by ``urd'', which is given by $ \frac{1+ \gamma}{3 + \gamma } ,$ by $ 1.$ 

For a given contour having $ 4n $ edges, with $ n \geq 3 ,$  its probability is therefore upper bounded by 
$$ \gamma^{ N(drd) + N(dru) + N (uru) } = \gamma^{n- N(urd) } \le \gamma^{n/2 - 1 }.$$

Finally, for $n=1,$ the probability of appearance of a contour of length $4$  is equal to $ \frac{1+ \gamma}{3 + \gamma } . $ Moreover, for $n=2,$  the probability of appearance of a contour of length $8 $ is equal to 
\begin{multline*}
 P ( D_1 = ulu, D_2 = urd, D_3 = drd) +P( D_1 = ulu, D_2 = uru, D_3 = urd ) +\\
 P ( D_1 = uru, D_2 = urd, D_3 = dld) + P ( D_1 = urd, D_2 = drd, D_3 = dld) \le 4 \gamma .
\end{multline*}

To conclude, notice that for each triplet we have $ 4 $ choices (the first entry of a given triplet is always fixed by the previous triplet in the sequence, and for $D_1, $ the first entry is always $u$). Therefore, a very rough upper bound on the total number of possible triplets $ (D_1, \ldots, D_{2n } ) $ is given by $  4^{2n} = 16^n .$ 
Therefore,  for all $ \gamma < \frac{1}{(16)^2 }, $
$$
P( \tau^i < \infty ) \le \frac{1+ \gamma}{3 + \gamma } + 4 \gamma + \sum\nolimits_{n\geq 3 }(16)^n \gamma^{n/2 - 1 } =  \frac{1+ \gamma}{3 + \gamma } + 4 \gamma  + (16)^2 \frac{ 16 \sqrt{\gamma}}{1 - 16 \sqrt{\gamma} }  .
$$
As $\gamma \to 0, $ the right hand side of the above inequality tends to $ \frac{1}{3} < 1 $ as $\gamma \to 0.$  As a consequence, $ \gamma_c > 0 .$ 

Finally, notice that by \eqref{eq:monotone}, the function   $ \gamma \mapsto P ( \lim_{t \to \infty } | C^{i, \gamma  } (t)|  = 0) $ is  increasing. Therefore, 
$$ \gamma_c = \inf \{ \gamma : P ( \lim_{t \to \infty } | C^{i, \gamma  } (t)|  = 0 )  = 1 \} ,$$
and the first two parts of the proof show that $ 0 < \gamma_c < \infty .$ 
This concludes the proof. 
\end{proof} 

\section{Proof of Theorem \ref{theo:0}}\label{sec:proof}
We only have to consider the case $ \gamma < \gamma_c;$ the case $ \gamma > \gamma_c $ is immediate. 

The proof follows from standard arguments (cf. e.g. Proof of Theorem 3.10 in Chapter VI of \cite{Liggett1985}).

We have to show that for each $i \in \Z, $ $\eta_i^{\bf 1 } ( t) = 1 $ infinitely often if $ \gamma < \gamma_c.$ It is sufficient to give the proof for $ i = 0.$ For that sake,  recall that by Corollary \ref{cor:1}, 
$$ P ( \eta^{\bf 1}_0 (t)=  1 ) = P ( \tau^0 > t ) \downarrow  P ( \tau^0 = \infty ) > 0 $$ 
as $ t \to \infty .$ 
Therefore, 
$$ A_t := \int_0^t P ( \eta^{\bf 1}_0 (s) = 1 ) ds \to \infty $$
as $ t \to \infty .$ Now let 
$$ \sigma_s := \inf \{ u\geq s : \eta^{\bf 1}_0 (u) = 1 \} .$$ Then
for $ s < t , $ and by the strong Markov property, 
\begin{multline*}
 A_t - A_s = E^{\bf 1} \int_s^t 1_{ \{ \eta_0 ( u) = 1 \} }du 
= E^{\bf 1} \left( \int_{\sigma_s}^t 1_{ \{ \eta_0 ( u) = 1 \} }du ; \sigma_s \le t \right) \\
=  E^{\bf 1} \left( 1_{\{\sigma_s \le t \} } E ( \int_{\sigma_s}^t 1_{ \{ \eta_0 ( u) = 1 \} }du | {\cal F}_{\sigma_s} )\right) \\
\le  E^{\bf 1} \left( 1_{\{\sigma_s \le t \} } E^{\eta^{\bf 1} (\sigma_s) } (\int_0^t  1_{ \{ \eta_0 ( u) = 1 \} }du) \right) \\
\le A_t P^{\bf 1 } (\sigma_s \le t) ,
\end{multline*}
where we have used that  
$$ E^{\eta^{\bf 1} (\sigma_s) } \int_0^t  1_{ \{ \eta_0 ( u) = 1 \} }du \le A_t .$$
As a consequence,
$$ P^{\bf 1 } (\sigma_s \le t) \geq 1 - \frac{ A_s}{A_t } , $$
whence 
$$ P^{\bf 1 } (\sigma_s < \infty ) = 1 .$$ 
In particular, we obtain that $ P^{\bf 1 } ( \forall n : \exists t \geq n : \eta_0 ( t) = 1 \} = 1 , $ that is, that $P^{\bf 1} ( \eta_0 (\cdot ) = 1  \mbox{ infinitely often } ) = 1.$

\section{Discussion of the case $\gamma = 0 $}\label{sec:4} If $ \gamma = 0, $ then the dual process $C^i (t) $ has the following very particular feature.  It either dies out at its first jump time or it survives forever. We write $ \tau^i  = \inf \{ t : C^i (t) = \emptyset \}$ for the extinction time of $ C^i.$ 
\begin{prop}
Let $T_1 $ be the first jump time of $C^i ( t) .$ Then $ \P ( T_1 = \tau^i ) = \frac13 , $ and on $ \{ T_1 \neq \tau^i \}, $ the process survives forever, that is, $\{ T_1 \neq \tau^i \} = \{ \tau^i = \infty \} ,$ $ \P(\tau^i = \infty) = \frac23 ,$ and on $ \{ \tau^i = \infty \}, $ $ \lim_{t \to \infty} C^i  (t) = \Z.$ 
\end{prop}

\begin{proof}
Either at time  $T_1 ,$ a transformation $  \pi_i^* $ is applied, in which case $ C^i ( T_1 ) = \emptyset $ and $T_1 = \tau^i . $ Or a transformation $  \pi_j^* $ for some $ j = i \pm 1$ is applied. Each of these transformations arrives at rate $1,$ hence the probability that $  \pi_i^* $ arrives first is $ \frac13. $ Finally, if $ \pi_j^* $ arrives first for $ j = i \pm 1, $ then $ C^i (T_1 )= \{ i , i \pm 1 \} .$ It is easy to prove that in this case, $ C^i ( t) $ is simply connected and strictly growing, that is, if $ j < k $ and $ j , k  \in C^i (t) ,$ then for all $ l $ with $ j \le l \le k , $ $ l \in C^i ( t+s ) $ for any $ s \geq 0 .$ Finally, write $ r (t) := \sup \{ j : j \in C^i (t) \} $ and $ l(t) = \inf \{ j : j \in C^i ( t) \} .$ Then $r (t) \to r(t) + 1 $ at rate $ 1 $ and $ l(t) \to l(t) - 1 $ at rate $1.$ 

\end{proof}
Let us come back to the original process in ``forward''-time, $(\eta (t), t \geq 0)  .$ This process is a pure spiking process without any leak effect. Write $ {\mathcal X} = \{ \xi \in \{ 0, 1 \}^\Z : ( 1 - \xi_i) ) ( 1 - \xi_{ i\pm 1 } ) = 0 \; \forall i \in \Z \} $  for the set of all configurations where no neighboring neurons have both potential $0.$ It is evident that $ {\mathcal X}$ is invariant under the evolution. 

In the following, we write $ P_t $ for the transition semigroup of $(\eta (t), t \geq 0 ) . $

\begin{prop}
Let $ \gamma = 0 $ and $  {\bf 1} $ be the configuration with $  {\bf 1}_i  = 1 $ for all $ i \in \Z.$ Then $\frac1t \int_0^t P_s ( {\bf 1} , \cdot ) ds $ converges to an invariant measure $ \mu_0 $ which concentrates on $ {\mathcal X}, $ that is, $ \mu_0 ( {\mathcal X}) = 1.$ $\mu_0 $ is the only invariant measure of $(\eta (t))_{t \geq 0} $ on $ \{0, 1 \}^\Z \setminus \{ \bf 0 \} , $ where $ \bf 0 $ denotes the all-zero configuration. Under $ \mu_0, $ the density of $1$ is given by $ 2/3 .$
\end{prop}

\begin{proof}
 By compactness of $ \{0, 1 \}^\Z , $ there exists a sequence $ t_n \to \infty $ and a measure $ \mu_0 $ such that 
$$ \frac{1}{t_n} \int_0^{t_n } P_s ({\bf 1 }, \cdot ) ds \to \mu_0 $$
as $n \to \infty .$ Moreover, $ \mu_0 $ is necessarily an invariant measure. Since each $ P_s ( {\bf 1}  , \cdot ) $ is supported by $ {\mathcal X}, $ obviously, $ \mu_0 ( {\mathcal X} ) = 1 $ as well. By duality, we have 
$$ P_s ({\bf 1 },  \{ \xi : \xi_i = 1 \} ) = \P ( C^i (s) \neq \emptyset ) = \P ( \tau^i > s )= e^{-3 s } + ( 1 - e^{-3s}) \frac23 \to \frac23 $$
as $ s \to \infty, $ whence $ \mu_0 ( \{ \xi : \xi_i  = 1 \} ) = \frac23.$  

We now prove that $ \mu_0 $ is the only such invariant measure. For that sake, let $ \mu $ be another invariant measure with $ \mu (\{{\bf 0 }\}) = 0 .$ By attractiveness, $ P_s ( {\bf 1}, \cdot ) \geq \mu P_s (\cdot ) $\footnote{Here, we write $ \mu_1 \le \mu_2 , $ if for all continuous monotone functions $ f : \{ 0, 1\}^\Z \to \R , $ $\int f d \mu_1 \le \int f d \mu_2 .$} for all $ s \geq 0, $ and therefore, $ \mu \le \mu_0 .$ By duality and since $ C^i ( \infty ) = \Z $ on $ \{ \tau^i = \infty \}, $ 
$$ \mu ( \{ \xi : \xi_i = 1 \} ) = \int \mu (d \xi ) \P ( \xi \cap C^i ( \infty ) \neq \emptyset) =  \frac23  \mu ( \{ \xi : \xi \neq {\bf 0} \} ) = \frac23.  $$ 
Thus, 
$$ \mu \le \mu_0 \mbox{ and for all } i, \, \mu ( \{ \xi : \xi_i = 1 \} ) = \mu_0 ( \{ \xi : \xi_i = 1 \} ).$$
This is only possible if $ \mu_0 = \mu$ (see e.g.\ Corollary 2.8, Chapter II of \cite{Liggett1985}). 
\end{proof}

Let us write $\eta^{\mu_0} (t) $ for the process starting in its invariant measure. 

\begin{cor}[Local Markov property]
Fix $i \in \Z.$ Then the process $( \eta^{\mu_0}_i (t) , t \geq 0 )$ is Markov, switching from $0$ to $1$ at rate $ \frac23 $ and from $ 1 $ to $0$ at rate $ 1.$ 
\end{cor}

\begin{proof}
If $ \eta_i^{\mu_0} ( t) = 1,$ then it spikes at rate $1, $ and thus switches from $1$ to $0, $ independently of the configuration of its neighbors. If $ \eta_i^{\mu_0} ( t) = 0 , $ then we know that the two neighbors $ i \pm 1 $ of site $i$ have value $1, $ since $ \mu_0 ( {\mathcal X}) = 1.$ Hence the rate of going from $ 0$ to $1$ is given by $ \frac23.$ 
\end{proof}

\section*{Acknowledgements}
Many thanks to Marzio Cassandro and Christophe Pouzat for illuminating discussions, constructive criticism and useful information on neurobiology.  This research has been conducted as part of
the project Labex MME-DII (ANR11-LBX-0023-01) and it is part of USP
project {\em Mathematics, computation, language and the brain} and of
FAPESP project {\em Research, Innovation and Dissemination Center for
  Neuromathematics} (grant 2013/07699-0). AG is partially supported by
CNPq fellowship (grant 311 719/2016-3.)

\bibliography{biblio}{}
\bibliographystyle{imsart-nameyear}

\end{document}